\documentclass[reqno]{amsart}
\usepackage{hyperref}

\AtBeginDocument{{}
\thanks{}
\vspace{8mm}}

\begin{document}
\title[$\psi $-Hilfer fractional
functional integrodifferential equation]
{Existence and Ulam-Hyers Mittag-Leffler stability of $\psi $-Hilfer fractional
functional integrodifferential equation}

\author[M. S. Abdo, A. M. Saeed, S. K. Panchal]
{Mohammed S. Abdo, Abdulkafi M. Saeed, Satish K. Panchal}

\address{Mohammed S. Abdo \newline
Department of Mathematics,
Dr.Babasaheb Ambedkar Marathwada University,
 Aurangabad, (M.S),431001, India}
\email{msabdo1977@gmail.com}

\address{ Abdulkafi M. Saeed  \newline
 Department of Mathematics,
 College of Science, Qassim University
 Buraydah-51452, Saudi Arabia}
\email{abdelkafe@yahoo.com}

\address{Satish K. Panchal \newline
Department of Mathematics,
Dr.Babasaheb Ambedkar Marathwada University,
 Aurangabad, (M.S),431001, India}
\email{drpanchalskk@gmail.com}

\dedicatory{A PREPRINT}

\thanks{}

\begin{abstract}
 This paper is committed to establishing the assumptions
essential for the existence and uniqueness results of a fractional
functional integrodifferential equation (FFIDE) having a derivative of
generalized Hilfer type. Using the Picard operator method, and Banach fixed
point theorem, we obtain the existence and uniqueness solution to the proposed
problem. Along with this, the Ulam-Hyers Mittag-Leffler (UHML)
stability is discussed via Pachpatte's inequality. For supporting our
results, an illustrative example will be introduced.
\end{abstract}

\maketitle
\numberwithin{equation}{section}
\newtheorem{theorem}{Theorem}[section]
\newtheorem{lemma}[theorem]{Lemma}
\newtheorem{remark}[theorem]{Remark}
\newtheorem{definition}[theorem]{Definition}
\newtheorem{example}[theorem]{Example}
\allowdisplaybreaks

\section{Introduction}

The theory of fractional differential equations is much significant due to
their nonlocal property is convenient to describe memory phenomena in many
applied fields such as biological sciences, physical sciences, economics,
engineering, and in fluid dynamic traffic model. The existence, stability
and control theory to fractional differential equations have been emerging
as an important area of investigation in the last few decades. For details,
we refer the reader to monographs of Samko \cite{SAO21}, Podlubny \cite{IP18}%
, Hilfer \cite{RH19}, Kilbas \cite{KL1}, and the papers \cite%
{AB,AH,AMB,FK,WZF} and the references therein.
For the recent review of the fractional calculus operators, see \cite%
{AR1,AH1,AMB,SO2,KL1,Yang1,Yang5,Yang6,Yang7}.

On the other hand, the stability results of functional differential
equations have been strongly developed. Very significant contributions about
this topic were introduced by Ulam \cite{Ul15}, Hyers \cite{Hy4} and this
type of stability called Ulam-Hyers stability. Thereafter improvement of
Ulam-Hyers stability provided by Rassias \cite{Ra10} in 1978.

Most recently a fractional derivative with kernel of function is introduced
by Almedia in \cite{AR1}, Da Sousa and De Oliveira in \cite{SO2}. The recent
development of $\psi $-fractional differential equations and the theoretical
analysis can be seen in \cite{AP1,AP2,HSB,KU,LWO,SAB,WZF}. For the recent
review of fractional functional differential equations, we will survey some
of the works as follows:

D. Otrocol, V. Ilea in \cite{OL} studied the Ulam--Hyers stability and
generalized Ulam--Hyers--Rassias stability for the following delay
differential equation

\begin{equation*}
\left\{
\begin{array}{c}
u^{\prime }(t)=f(t,u(t),u(h(t)),\text{ \ }t\in \lbrack a,b], \\
u(t)=\psi (t),\qquad t\in \lbrack a-h,a]. \ \ \ \
\end{array}%
\right.
\end{equation*}

J. Wang and Y. Zhang \cite{WZ}, proved some results of existence,
uniqueness, and Ulam--Hyers--Mittag-Leffler stable of Caputo-type
fractional-order delay differential equation

\begin{equation}
\left\{
\begin{array}{c}
^{C}D_{0^{+}}^{\alpha }u(t)=f(t,u(t),u(h(t)),\text{ \ }t\in \lbrack 0,d], \\
u(t)=\psi (t),\qquad t\in \lbrack -h,0].\qquad \ \ \ \ \
\end{array}%
\right.  \label{wz}
\end{equation}

Liu et al. in \cite{LWO} established the existence, uniqueness, and
Ulam--Hyers--Mittag-Leffler stability of solutions to a class of $\psi $%
-Hilfer fractional-order delay differential equations%
\begin{equation}
\left\{
\begin{array}{c}
^{H}D_{0^{+}}^{\alpha ,\beta ;\psi }u(t)=f(t,u(t),u(h(t)),\text{ \ }t\in
(0,d], \\
I_{0^{+}}^{1-\gamma ;\psi }u(0^{+})=u_{0}\in
\mathbb{R}
,\qquad \qquad \ \qquad \qquad \\
u(t)=\varphi (t),\qquad t\in \lbrack -h,0].\qquad \qquad \ \ \ \ \ \ \
\end{array}%
\right.  \label{lwo}
\end{equation}%
K.D. Kucche, and P.U. Shikhare in \cite{KU1} studied the existence,
uniqueness of a solution and Ulam type stabilities for Volterra delay
integro-differential equations on a finite interval%
\begin{equation}
\left\{
\begin{array}{c}
u^{\prime }(t)=f\left(
t,y(t),y(g(t)),\int_{0}^{t}h(t,s,y(s),y(g(s))ds\right) ,\text{ \ }t\in
\lbrack 0,b], \\
u(t)=\varphi (t),\qquad t\in \lbrack -r,0],\ 0<r<\infty ,%
\end{array}%
\right.  \label{ku1}
\end{equation}

Motivated by aforesaid works, in this paper, we establish the existence,
uniqueness and UHML stability of solutions for $\psi $-Hilfer
fractional-order functional integrodifferential equations of the form:

\begin{equation}
\left\{
\begin{array}{c}
^{H}D_{0^{+}}^{\alpha ,\beta ;\psi }u(t)=f\left(
t,y(t),y(g(t)),\int_{0}^{t}h(t,s,y(s),y(g(s))ds\right) ,\text{ \ }t\in (0,b],
\\
I_{0^{+}}^{1-\gamma ;\psi }u(0^{+})=u_{0},\qquad 0<\gamma \leq 1\qquad
\qquad \qquad \\
u(t)=\varphi (t),\qquad t\in \lbrack -r,0],\ 0<r<\infty ,\qquad \qquad%
\end{array}%
\right.  \label{equ 1}
\end{equation}%
where $0<\alpha <1,$ $0\leq \beta \leq 1,$, $^{H}D_{0^{+}}^{\alpha ,\beta
;\psi }(\cdot )$ and $I_{0^{+}}^{1-\gamma ;\psi }(\cdot )$ are $\psi -$%
Hilfer fractional derivative of order$\ \left( \alpha ,\beta \right) ,$ and $%
\psi -$Riemann--Liouville fractional integral of order $1-\gamma $ $(\gamma
=\alpha +\beta (1-\alpha ),$ respectvely$,$ $\varphi \in C([-r,0],%
\mathbb{R}
),$ $f:[0,b]\times \mathbb{R}\times \mathbb{R}\times \mathbb{R}\rightarrow
\mathbb{R}$, $h:[0,b]\times \lbrack 0,b]\times \mathbb{R}\rightarrow \mathbb{%
R}$ and $g:[0,b]\rightarrow \lbrack -r,0]$ are continuous functions, and $%
g(t)\leq t$.

We apply Picard's operator method, Banach fixed point theorem, and the
Pachpatte's inequality to achieve our results. The results obtained in this
paper are more general than the known results and include the study of \cite%
{CR,OL,WZ,LWO,KU1} as special cases of (\ref{equ 1}).

The main contributions are as follows: In section 2, some preliminary
results and notations are provided which are useful in the sequel. In
Section 3, we study the existence and uniqueness results on the problem (\ref%
{equ 1}) by means of Banach fixed point theorem and Picard operator method.
Section 4 is devoted to discussing the UHML stability result via Pachpatte's
inequality. Finally, an illustrative example is provided in the last section.

\section{\textbf{Preliminaries\label{Sec2}}}

In this section, we will present some preliminaries and lemmas of fractional
calculus theory and nonlinear analysis which are used in this paper. Let $%
\left[ a,b\right] \subset
\mathbb{R}
^{+}$ with $(0<a<b<\infty )$ and let $C\left[ a,b\right] $ be the space of
continuous function, $\omega :\left[ a,b\right] \rightarrow
\mathbb{R}
$ with the norm $\left\Vert \omega \right\Vert _{C}=\max \{\left\vert \omega
(t)\right\vert :a\leq t\leq b\}.$ We consider the weighted spaces $%
C_{1-\gamma ;\psi }[a,b]$ as follows \newline
\begin{equation*}
C_{1-\gamma ;\psi }\left[ a,b\right] =\left\{ \omega :(a,b]\rightarrow
\mathbb{R}
;\text{ }\left[ \psi (t)-\psi (a)\right] ^{1-\gamma }\omega (t)\in C\left[
a,b\right] \right\} ,
\end{equation*}%
\ where $0<\gamma <1,$ $n\in
\mathbb{N}
,$ with the norm%
\begin{equation*}
\left\Vert \omega \right\Vert _{c_{1-\gamma ;\psi }}=\underset{t\in \left[
a,b\right] }{\max }\left\vert \left[ \psi (t)-\psi (a)\right] ^{1-\gamma
}\omega (t)\right\vert ,
\end{equation*}%
for $0<\gamma <1,$ $\delta \geq 0.$ Denote $E_{\alpha }(\cdot )\ $and $%
E_{\alpha ,\beta }(\cdot )$ by the Mittag-Leffler functions defined by

\begin{equation*}
E_{\alpha }(\omega )=\sum\limits_{k=0}^{\infty }\frac{\omega ^{k}}{\Gamma
(\alpha k+1)},\ \ \omega \in
\mathbb{C}
,\ \ \Re (\alpha )>0.
\end{equation*}%
\begin{equation*}
E_{\alpha ,\beta }(\omega )=\sum\limits_{k=0}^{\infty }\frac{\omega ^{k}}{%
\Gamma (\alpha k+\beta )},\ \ \omega \in
\mathbb{C}
,\ \ \Re (\alpha ),\Re (\beta )>0.
\end{equation*}

\begin{definition}
\label{d1} \cite{KL1,AR1} Let $\alpha >0$ be a real number and $\omega
:[a,b]\rightarrow
\mathbb{R}
$\ a function. Given another function $\psi \in C^{1}\left[ a,b\right] $ be
an increasing having a continuous derivative $\psi ^{\prime }$ on $(a,b)$.
Then

The left-sided $\psi -$Riemann-Liouville fractional integral of $\omega ,$
of order $\alpha $ is defined by%
\begin{equation*}
I_{a^{+}}^{\alpha ,\psi }\omega (t)=\frac{1}{\Gamma (\alpha )}%
\int_{a}^{t}\psi ^{\prime }(s)(\psi (t)-\psi (s))^{\alpha -1}\omega (s)ds.
\end{equation*}%
The left-sided $\psi -$Riemann-Liouville fractional derivative of $\omega
\in C^{n}[a,b]$ of order $\alpha $ is defined by%
\begin{equation*}
D_{a^{+}}^{\alpha ,\psi }\omega (t)=\left( \frac{1}{\psi ^{\prime }(t)}\frac{%
d}{dt}\right) ^{n}I_{a^{+}}^{n-\alpha ,\psi }\omega (t),\text{ \ }n=[\alpha
]+1.
\end{equation*}%
The left-sided $\psi -$Caputo fractional derivative of $\omega ,$ of order $%
\alpha $ is defined by%
\begin{equation*}
^{C}D_{a^{+}}^{\alpha ,\psi }\omega (t)=D_{a^{+}}^{\alpha ,\psi }\left(
\omega (t)-\sum_{k=0}^{n-1}\frac{\omega _{\psi }^{[k]}(a)}{k!}(\psi (t)-\psi
(s))^{k}\right) ,
\end{equation*}%
where $n=[\alpha ]+1$\ for $\alpha \notin
\mathbb{N}
,$ and $\omega _{\psi }^{[k]}(t)=\left( \frac{1}{\psi ^{\prime }(t)}\frac{d}{%
dt}\right) ^{k}\omega (t).$ In particular, if $n=\alpha $, we have $%
^{C}D_{a^{+}}^{\alpha ,\psi }\omega (t)=\omega _{\psi }^{[n]}(t).$

\begin{lemma}
\cite{KL1} Let $\alpha >0$ and $\beta >0$. Then, we have the following
semigroup property given by%
\begin{equation*}
I_{a^{+}}^{\alpha ,\psi }I_{a^{+}}^{\beta ,\psi }(\cdot )=I_{a^{+}}^{\alpha
+\beta ,\psi }(\cdot )
\end{equation*}
\end{lemma}
\end{definition}

\begin{definition}
\textbf{\label{de3}} \cite{SO2} Let $n-1<\alpha <n$ $(n\in
\mathbb{N}
)$, and $\omega ,\psi \in C^{n}[a,b]$ such that $\psi $ is an increasing
with $\psi ^{\prime }(t)\neq 0$ for all $t\in \lbrack a,b].$ Then the
left-sided $\psi $-Hilfer fractional derivative of $\omega $ of order $%
\alpha $ and type $0\leq \beta \leq 1$ is defined by%
\begin{equation*}
^{H}D_{a^{+}}^{\alpha ,\beta ,\psi }\omega (t)=I_{a^{+}}^{\beta (n-\alpha
);\psi }\left( \frac{1}{\psi ^{\prime }(t)}\frac{d}{dt}\right)
^{n}I_{a^{+}}^{(1-\beta )(n-\alpha );\psi }\omega (t).
\end{equation*}%
One has,
\begin{equation*}
^{H}D_{a^{+}}^{\alpha ,\beta ,\psi }\omega (t)=I_{a^{+}}^{\beta (n-\alpha
);\psi }D_{a^{+}}^{\gamma ;\psi }\omega (t),
\end{equation*}%
where
\begin{equation*}
D_{a^{+}}^{\gamma ;\psi }\omega (t)=\left( \frac{1}{\psi ^{\prime }(t)}\frac{%
d}{dt}\right) ^{n}I_{a^{+}}^{n-\gamma ;\psi }\omega (t),\text{ }\gamma
=\alpha +\beta (n-\alpha ).
\end{equation*}
\end{definition}

\begin{remark}
\label{re2} From Definition \ref{de3}, we observe that, if $0<\alpha <1,$ $%
0\leq \beta \leq 1$ and $\gamma =\alpha +\beta (1-\alpha ),$ then
\begin{equation*}
^{H}D_{a^{+}}^{\alpha ,\beta ,\psi }\omega (t)=I_{a^{+}}^{\beta (1-\alpha
);\psi }\left( \frac{1}{\psi ^{\prime }(t)}\frac{d}{dt}\right)
I_{a^{+}}^{(1-\beta )(1-\alpha );\psi }\omega (t).
\end{equation*}%
One has,
\begin{equation*}
^{H}D_{a^{+}}^{\alpha ,\beta ,\psi }\omega (t)=I_{a^{+}}^{\beta (1-\alpha
);\psi }D_{a^{+}}^{\gamma ;\psi }\omega (t)=I_{a^{+}}^{\gamma -\alpha ;\psi
}D_{a^{+}}^{\gamma ;\psi }\omega (t),
\end{equation*}%
where
\begin{equation*}
D_{a^{+}}^{\gamma ;\psi }\omega (t)=\left( \frac{1}{\psi ^{\prime }(t)}\frac{%
d}{dt}\right) I_{a^{+}}^{(1-\beta )(1-\alpha );\psi }\omega (t).
\end{equation*}%
Now, we introduce the weighted spaces \newline
\begin{equation*}
C_{1-\gamma ;\psi }^{\alpha ,\beta }[a,b]=\{\omega \in C_{1-\gamma ;\psi
}[a,b],D_{a^{+}}^{\alpha ,\beta ;\psi }\omega \in C_{1-\gamma ;\psi }[a,b]\},
\end{equation*}%
and
\begin{equation}
C_{1-\gamma ;\psi }^{\gamma }[a,b]=\{\omega \in C_{1-\gamma ;\psi
}[a,b],D_{a^{+}}^{\gamma ;\psi }\omega \in C_{1-\gamma ;\psi }[a,b]\},
\label{a1}
\end{equation}%
where $0<\gamma <1.$ Since $D_{a^{+}}^{\alpha ,\beta ;\psi }\omega
=I_{a^{+}}^{\beta (1-\alpha );\psi }D_{a^{+}}^{\gamma ;\psi }\omega ,$ it is
obvious that, $C_{1-\gamma ;\psi }^{\gamma }[a,b]\subset C_{1-\gamma ;\psi
}^{\alpha ,\beta }[a,b].$
\end{remark}

\begin{lemma}
\label{def8.8} \cite{AP1} Let $\alpha >0$, $\beta >0$ and $\gamma =\alpha
+\beta (1-\alpha ).$ If $\omega \in C_{1-\gamma ;\psi }^{\gamma }[J,%
\mathbb{R}
]$, then\newline
\begin{equation*}
I_{a^{+}}^{\gamma ;\psi }D_{a^{+}}^{\gamma ;\psi }\omega =I_{a^{+}}^{\alpha
;\psi }D_{a^{+}}^{\alpha ,\beta ;\psi }\omega ,
\end{equation*}%
and
\begin{equation*}
D_{a^{+}}^{\gamma ;\psi }I_{a^{+}}^{\alpha ;\psi }\omega =D_{a^{+}}^{\beta
(1-\alpha );\psi }\omega .
\end{equation*}
\end{lemma}

\begin{theorem}
\label{th2.3} \cite{SO2} Let $\omega \in C^{1}[a,b],$ $0<\alpha <1$, and $%
0\leq \beta \leq 1$. Then%
\begin{equation*}
^{H}D_{a^{+}}^{\alpha ,\beta ,\psi }I_{a^{+}}^{\alpha ,\psi }\omega
(t)=\omega (t).
\end{equation*}
\end{theorem}

\begin{theorem}
\label{th2.3a} \cite{KL1,SO2} Let $\alpha ,\sigma >0$, and $0\leq \beta \leq
1.$ Then
\end{theorem}

\begin{equation*}
I_{a^{+}}^{\alpha ,\psi }\left[ \psi (t)-\psi (a)\right] ^{\sigma -1}=\frac{%
\Gamma (\sigma )}{\Gamma (\alpha +\sigma )}(\psi (t)-\psi (a))^{\alpha
+\sigma -1}
\end{equation*}%
and%
\begin{equation*}
^{H}D_{a^{+}}^{\alpha ,\beta ,\psi }\left[ \psi (t)-\psi (a)\right] ^{\sigma
-1}=\frac{\Gamma (\sigma )}{\Gamma (\sigma -\alpha )}(\psi (t)-\psi
(a))^{\sigma -\alpha -1}.
\end{equation*}

\begin{theorem}
\label{th2.4} \cite{SO2} If $0<\alpha <1$, $0\leq \beta \leq 1$, $0<\gamma
<1\ $and that $\ \omega \in C_{1-\gamma }[a,b],$ $I_{a^{+}}^{1-\gamma ;\psi
}\omega \in C_{1-\gamma }^{1}[a,b],$ then%
\begin{equation*}
I_{a^{+}}^{\alpha ;\psi }\ ^{H}D_{a^{+}}^{\alpha ,\beta ,\psi }\omega
(t)=\omega (t)-\frac{I_{0^{+}}^{1-\gamma ;\psi }\omega (a)}{\Gamma (\gamma )}%
(\psi (t)-\psi (a))^{\gamma -1}.
\end{equation*}
\end{theorem}

\begin{theorem}
\label{th2.5} \cite{SO2} Let $\omega \in C_{\gamma }[a,b],$ $0<\gamma
<\alpha <1$. Then we have%
\begin{equation*}
I_{a^{+}}^{\alpha ;\psi }\ \omega (a)=\underset{t\longrightarrow a^{+}}{\lim
}I_{a^{+}}^{\alpha ;\psi }\omega (t)=0.
\end{equation*}
\end{theorem}

\begin{definition}
\textbf{\label{de2.5}} \cite{WZ} Let $(X,d)$ be a metric space. Now $%
T:X\longrightarrow X$ is a Picard operator if there exists $u^{\ast }\in X$\
such that $F_{T}=$ $u^{\ast }$ where $F_{T}$ $=\{u\in X:T(u)=u\}$ is the
fixed point set of $T$, and the sequence $(T^{n}(u_{0}))_{n\in
\mathbb{N}
}$ converges to $u^{\ast }$ for all $u_{0}\in X$.
\end{definition}

\begin{lemma}
\textbf{\label{le2.6}} \cite{WZ} Let $(X,d,\leq )$ be an ordered metric
space, and let $T:X\longrightarrow X$ be an increasing Picard operator with $%
F_{T}$ = $\{u_{T}^{\ast }\}$. Then for $u\in X,$ $u\leq T(u)$ implies$\
u\leq u_{T}^{\ast }.$
\end{lemma}

\begin{lemma}
\label{tt} (\cite{GBP}) (Pachpatte's inequality). Let $x(t),$ $p(t)$ and $%
q(t)$ be nonnegative continuous functions defined on $%
\mathbb{R}
^{+}$,and $\eta (t)$ be a positive and nondecreasing continuous function
defined on $%
\mathbb{R}
^{+}$ for which the inequality%
\begin{equation*}
x(t)\leq \eta (t)+\int_{0}^{t}p(s)\left[ x(s)+\int_{0}^{s}q(\sigma )x(\sigma
)d\sigma \right] ds,
\end{equation*}%
holds for $t\in
\mathbb{R}
^{+}$. Then%
\begin{equation}
x(t)\leq \eta (t)\left[ 1+\int_{0}^{t}p(s)\exp \int_{0}^{s}\left[ p(\sigma
)+q(\sigma )\right] d\sigma \right] ds,  \label{ht}
\end{equation}%
for $t\in
\mathbb{R}
^{+}.$
\end{lemma}

\textbf{Lemma 2.7} \label{lem2.7} Let $f:(0,b]\times \mathbb{R}\times
\mathbb{R}\times \mathbb{R}\rightarrow \mathbb{R}$ be a continuous function.
Then the problem
\begin{equation*}
\begin{array}{c}
^{H}D_{0^{+}}^{\alpha ,\beta ;\psi }u(t)=F_{u,g,h}(s),\text{ }t\in (0,b] \\
I_{0^{+}}^{1-\gamma ;\psi }u(0)=u_{0},\qquad \ \qquad \qquad%
\end{array}%
\end{equation*}%
is equivalent to integral equation%
\begin{equation}
u(t)=\mathcal{H}_{\psi }^{\gamma }(t,0)u_{0}+\frac{1}{\Gamma (\alpha )}%
\int_{0}^{t}\mathcal{N}_{\psi }^{\alpha }(t,s)F_{u,g,h}(s)ds,  \label{equ 7}
\end{equation}%
where $\mathcal{H}_{\psi }^{\gamma }(t,0):=\frac{(\psi (t)-\psi (0))^{\gamma
-1}}{\Gamma (\gamma )}$, $\mathcal{N}_{\psi }^{\alpha }(t,s):=\psi ^{\prime
}(s)(\psi (t)-\psi (s))^{\alpha -1},$and
\begin{equation*}
F_{u,g,h}(s):=f\left( s,u(s),u(g(s)),\int_{0}^{s}h(s,\tau ,u(\tau ),u(g(\tau
))d\tau \right) .
\end{equation*}

\section{\textbf{Main results \label{Sec3}}}

In this section, we present results on the existence, uniqueness, and UHML
stability of solutions to the problem (\ref{equ 1}). First, we introduce the
following hypotheses:

\begin{description}
\item[(H$_{1}$)] $g:(0,b]\rightarrow \lbrack -r,0]$ is continuous function
with $g(t)\leq t$.

\item[(H$_{2}$)] $f:(0,b]\times \mathbb{R}\times \mathbb{R}\times \mathbb{R}%
\rightarrow \mathbb{R}$, $h:(0,b]\times (0,b]\times \mathbb{R}\rightarrow
\mathbb{R}$ are two continuous functions, and there exist $L_{F},L_{h}>0$
such that
\end{description}

\begin{equation*}
\left\vert f(t,u_{1},u_{2},u_{3})-f(t,v_{1},v_{2},v_{3})\right\vert \leq
L_{f}\text{ }\left[ \left\vert u_{1}-v_{1}\right\vert \text{ }+\left\vert
u_{2}-v_{2}\right\vert +\left\vert u_{3}-v_{3}\right\vert \right] ,
\end{equation*}%
\begin{equation*}
\left\vert h(t,u_{1},u_{2})-h(t,v_{1},v_{2})\right\vert \leq L_{h}\text{ }%
\left[ \left\vert u_{1}-v_{1}\right\vert \text{ }+\left\vert
u_{2}-v_{2}\right\vert \right] ,
\end{equation*}%
for\ all\ $t\in (0,b],$ $u_{i},v_{i}\in
\mathbb{R}
,i=1,2,3.$

\begin{description}
\item[$($H$_{3})$] The\ following inequality holds%
\begin{equation*}
\Theta :=2L_{f}\left( \frac{\mathcal{B}(\gamma ,\alpha )}{\Gamma (\alpha )}+%
\frac{L_{h}}{\zeta \gamma }\frac{\mathcal{B}(\gamma +1,\alpha )}{\Gamma
(\alpha )}\right) \left[ \psi (b)-\psi (0)\right] ^{\alpha +1}<1,
\end{equation*}%
where $\mathcal{B}(\cdot ,\cdot )$ is a beta function and $\zeta =\sup_{s\in
(0,b]}\left\vert \psi ^{\prime }(s)\right\vert .$

Next, before starting and proving our results, we need to the following
remarks.
\end{description}

\begin{remark}
\label{re4.3}A function $v\in C_{1-\gamma ,\psi }[0,b]$ is a solution of the
inequality%
\begin{equation}
\left\vert ^{H}D_{0^{+}}^{\alpha ,\beta ,\psi }v(t)-F_{v,g,h}(t)\right\vert
\leq \varepsilon E_{\alpha }(\psi (t)-\psi (0))^{\alpha },\ t\in (0,b],
\label{equ 5}
\end{equation}%
if and only if there exists a function $\eta _{v}\in C_{1-\gamma ,\psi
}[0,b] $ such that
\end{remark}

\begin{description}
\item[(i)] $\left\vert \eta _{v}(t)\right\vert \leq \varepsilon E_{\alpha
}((\psi (t)-\psi (0))^{\alpha }),\ \ t\in (0,b];$

\item[(ii)] $^{H}D_{0^{+}}^{\alpha ,\beta ,\psi }v(t)=F_{v,g,h}(t)+\eta
_{v}(t),$ $\ t\in (0,b],$ where
\begin{equation*}
F_{v,g,h}(t):=f\left( t,v(t),v(g(t)),\int_{0}^{t}h(t,s,v(s),v(g(s))ds\right)
.
\end{equation*}
\end{description}

\begin{definition}
\textbf{\label{de4.2}} Problem (\ref{equ 1}) is UHML stable with respect to $%
E_{\alpha }((\psi (t)-\psi (0))^{\alpha })$ if there exists $C_{_{E_{\alpha
}}}>0$ such that, for each $\varepsilon >0$ and for each solution $v\in
C[-r,b]$ to the inequality (\ref{equ 5}), there exists a solution $u\in $ $%
C[-r,b]$ to first equation of (\ref{equ 1}) with
\end{definition}

\begin{equation*}
\left\vert v(t)-u(t)\right\vert \leq C_{_{E_{\alpha }}}\varepsilon E_{\alpha
}((\psi (t)-\psi (0))^{\alpha }),\ \ t\in \left[ -r,b\right] .
\end{equation*}%
By Lemma \ref{lem2.7} and above remark, for $t\in (0,b]$ we have
\begin{eqnarray*}
v(t) &=&\mathcal{H}_{\psi }^{\gamma }(t,0)u_{0}+\frac{1}{\Gamma (\alpha )}%
\int_{0}^{t}\mathcal{N}_{\psi }^{\alpha }(t,s)F_{v,g,h}(s)ds \\
&&+\frac{1}{\Gamma (\alpha )}\int_{0}^{t}\mathcal{N}_{\psi }^{\alpha
}(t,s)\eta _{v}(s)ds.
\end{eqnarray*}

\begin{remark}
\label{h1} Let $v\in C_{1-\gamma ,\psi }[0,b]$ be a solution of the
inequality (\ref{equ 5}). Then $v$ is a solution of the following integral
inequality%
\begin{eqnarray*}
\left\vert v(t)-\mathcal{H}_{\psi }^{\gamma }(t,0)u_{0}-\frac{1}{\Gamma
(\alpha )}\int_{0}^{t}\mathcal{N}_{\psi }^{\alpha
}(t,s)F_{v,g,h}(s)ds\right\vert  &\leq &\frac{1}{\Gamma (\alpha )}%
\int_{0}^{t}\mathcal{N}_{\psi }^{\alpha }(t,s)\left\vert \eta
_{v}(s)\right\vert ds \\
&\leq &\frac{\varepsilon }{\Gamma (\alpha )}\int_{0}^{t}\mathcal{N}_{\psi
}^{\alpha }(t,s)E_{\alpha }((\psi (s)-\psi (0))^{\alpha })ds \\
&=&\varepsilon \sum\limits_{k=0}^{\infty }\frac{1}{\Gamma ((k+1)\alpha +1)}%
I_{0^{+}}^{\alpha ;\psi }\left[ \psi (s)-\psi (0)\right] ^{\alpha k} \\
&=&\varepsilon \sum\limits_{k=0}^{\infty }\frac{\left[ \psi (t)-\psi (0)%
\right] ^{\alpha (k+1)}}{\Gamma ((k+1)\alpha +1)} \\
&\leq &\varepsilon \sum\limits_{k=0}^{\infty }\frac{\left( \left[ \psi
(t)-\psi (0)\right] ^{\alpha }\right) ^{n}}{\Gamma (n\alpha +1)} \\
&=&\varepsilon E_{\alpha }(\left[ \psi (t)-\psi (0)\right] ^{\alpha }).
\end{eqnarray*}%
Now, we are ready to prove our results on the problem (\ref{equ 1}).
\end{remark}

\begin{theorem}
\label{Th3.1} Assume that (H$_{1}$)-(H$_{3}$) are fulfilled. Then

\begin{enumerate}
\item The $\psi -$Hilfer problem (\ref{equ 1}) has a unique solution in $%
C[-r,b]\cap C_{1-\gamma ;\psi }[0,b]$.

\item The first equation of (\ref{equ 1}) is UHML stable.
\end{enumerate}
\end{theorem}

\begin{proof}
(1) In view of Lemma \ref{lem2.7}, we get that (\ref{equ 1}) is equivalent
to the following system%
\begin{equation}
u(t)=\left\{
\begin{array}{c}
\mathcal{H}_{\psi }^{\gamma }(t,0)u_{0}+\frac{1}{\Gamma (\alpha )}%
\int_{0}^{t}\mathcal{N}_{\psi }^{\alpha }(t,s)F_{u,g,h}(s)ds,\ \ t\in
\lbrack 0,b], \\
\varphi (t)\qquad \ \qquad \qquad \qquad \qquad \qquad \ \ \qquad t\in
\lbrack -r,0].%
\end{array}%
\right.  \label{eq3}
\end{equation}%
where $\mathcal{H}_{\psi }^{\gamma }(t,0):=\frac{(\psi (t)-\psi (0))^{\gamma
-1}}{\Gamma (\gamma )}$, $\mathcal{N}_{\psi }^{\alpha }(t,s):=\psi ^{\prime
}(s)(\psi (t)-\psi (s))^{\alpha -1},$and
\begin{equation}
F_{u,g,h}(s):=f\left( s,u(s),u(g(s)),\int_{0}^{s}h(s,\tau ,u(\tau ),u(g(\tau
))d\tau \right) .  \label{p1}
\end{equation}

The existence of a solution for the problem (\ref{equ 1}) can be transformed
into a fixed point problem in $C[-r,b]$ for the operator $\mathcal{G}%
_{f}:C[-r,b]\longrightarrow C[-r,b]$\ defined by
\begin{equation}
\mathcal{G}_{f}u(t)=\left\{
\begin{array}{c}
\mathcal{H}_{\psi }^{\gamma }(t,0)u_{0}+\frac{1}{\Gamma (\alpha )}%
\int_{0}^{t}\mathcal{N}_{\psi }^{\alpha }(t,s)F_{u,g,h}(s)ds,\ \ t\in
\lbrack 0,b], \\
\varphi (t)\qquad \ \qquad \qquad \qquad \qquad \qquad \ \ \qquad t\in
\lbrack -r,0].%
\end{array}%
\right.  \label{equ9}
\end{equation}%
We remark that for any continuous function $F_{u,g,h},$ the operator $%
\mathcal{G}_{f}\mathcal{\ }$is also continuous. Indeed,

Case 1. For all $t,t+\epsilon \in (0,b],$ we have%
\begin{eqnarray*}
\left\vert \mathcal{G}_{f}u(t+\epsilon )-\mathcal{G}_{f}u(t)\right\vert
&=&\left\vert \mathcal{H}_{\psi }^{\gamma }(t+\epsilon ,0)u_{0}+\frac{1}{%
\Gamma (\alpha )}\int_{0}^{t+\epsilon }\mathcal{N}_{\psi }^{\alpha
}(t+\epsilon ,s)F_{u,g,h}(s)ds\right. \\
&&\left. -\mathcal{H}_{\psi }^{\gamma }(t,0)u_{0}-\frac{1}{\Gamma (\alpha )}%
\int_{0}^{t}\mathcal{N}_{\psi }^{\alpha }(t,s)F_{u,g,h}(s)ds\right\vert \\
&\rightarrow &0\text{ as }\epsilon \longrightarrow 0.
\end{eqnarray*}%
Case 2. For all $t,t+\epsilon \in C[-r,0],$ we have%
\begin{equation*}
\left\vert \mathcal{G}_{f}u(t+\epsilon )-\mathcal{G}_{f}u(t)\right\vert
=\left\vert \varphi (t+\epsilon )-\varphi (t)\right\vert \longrightarrow 0,%
\text{ as }\epsilon \longrightarrow 0.
\end{equation*}%
Next, we show that $\mathcal{G}_{f}:C[-r,b]\rightarrow C[-r,b]$ defined by (%
\ref{equ9}) is a contraction mapping on $C[-r,b]$\ with respect to\ the
weighted norm $\left\Vert \cdot \right\Vert _{C_{1-\gamma ;\psi }}.$

Case 1. For all $t\in \lbrack -r,0],$and for each $u,v\in C\left[ -r,b\right]
,$ we have%
\begin{equation*}
\left\vert \mathcal{G}_{f}u(t)-\mathcal{G}_{f}v(t)\right\vert =0.
\end{equation*}%
Case 2. From our assumption (\ref{p1}), and for each $t\in (0,b],$ $u,v\in
C_{1-\gamma ,\psi }\left[ 0,b\right] ,$ we have%
\begin{eqnarray}
\left\vert F_{u,g,h}(s)-F_{v,g,h}(s)\right\vert &=&\left\vert f\left(
s,u(s),u(g(s)),\int_{0}^{s}h(s,\tau ,u(\tau ),u(g(\tau ))d\tau \right)
\right.  \notag \\
&&\left. -f\left( s,v(s),v(g(s)),\int_{0}^{s}h(s,\tau ,v(\tau ),v(g(\tau
))d\tau \right) \right\vert  \notag \\
&\leq &L_{f}\left[
\begin{array}{c}
\left\vert u(s)-v(s)\right\vert \text{ }+\left\vert u(g(s)-v(g(s)\right\vert
\qquad \qquad \qquad \qquad \\
+\int_{0}^{s}\left\vert h(s,\tau ,u(\tau ),u(g(\tau ))-h(s,\tau ,v(\tau
),v(g(\tau ))\right\vert d\tau%
\end{array}%
\right]  \notag \\
&\leq &2L_{f}\left[ \left[ \psi (s)-\psi (0)\right] ^{\gamma -1}\left\Vert
u-v\right\Vert _{C_{1-\gamma ;\psi }\left[ 0,b\right] }\right]  \notag \\
&&+2L_{f}L_{h}\int_{0}^{s}\text{ }\left[ \left[ \psi (\tau )-\psi (0)\right]
^{\gamma -1}\left\Vert u-v\right\Vert _{C_{1-\gamma ;\psi }\left[ 0,b\right]
}\right] d\tau .  \label{s1}
\end{eqnarray}

Since $\psi \in C^{1}[0,b],$ there exists a constant $\zeta \neq 0$ such
that $\underset{\tau \in (0,b]}{sup}\left\vert \psi ^{\prime }(\tau
)\right\vert \leq \zeta .$ Therefore%
\begin{eqnarray}
&&\int_{0}^{s}\text{ }\left[ \left[ \psi (\tau )-\psi (0)\right] ^{\gamma
-1}\left\Vert u-v\right\Vert _{C_{1-\gamma ;\psi }\left[ 0,b\right] }\right]
d\tau   \notag \\
&=&\int_{0}^{s}\text{ }\left[ \psi ^{\prime }(\tau )\left[ \psi (\tau )-\psi
(0)\right] ^{\gamma -1}\left[ \psi ^{\prime }(\tau )\right] ^{-1}\left\Vert
u-v\right\Vert _{C_{1-\gamma ;\psi }\left[ 0,b\right] }\right] d\tau   \notag
\\
&\leq &\frac{1}{\zeta }\int_{0}^{s}\text{ }\left[ \psi ^{\prime }(\tau )%
\left[ \psi (\tau )-\psi (0)\right] ^{\gamma -1}\left\Vert u-v\right\Vert
_{C_{1-\gamma ;\psi }\left[ 0,b\right] }\right] d\tau   \notag \\
&=&\frac{1}{\zeta \gamma }\text{ }\left[ \left[ \psi (s)-\psi (0)\right]
^{\gamma }\left\Vert u-v\right\Vert _{C_{1-\gamma ;\psi }\left[ 0,b\right] }%
\right] .  \label{s2}
\end{eqnarray}%
The equations (\ref{s1}) and (\ref{s2}), gives%
\begin{eqnarray*}
\left\vert F_{u,g,h}(s)-F_{v,g,h}(s)\right\vert  &\leq &2L_{f}\left[ \left[
\psi (s)-\psi (0)\right] ^{\gamma -1}\left\Vert u-v\right\Vert _{C_{1-\gamma
;\psi }\left[ 0,b\right] }\right]  \\
&&+2L_{f}L_{h}\frac{1}{\zeta \gamma }\text{ }\left[ \left[ \psi (s)-\psi (0)%
\right] ^{\gamma }\left\Vert u-v\right\Vert _{C_{1-\gamma ;\psi }\left[ 0,b%
\right] }\right] .
\end{eqnarray*}%
Consequently,
\begin{eqnarray*}
\left\vert \mathcal{G}_{f}u(t)-\mathcal{G}_{f}v(t)\right\vert  &\leq &\frac{%
2L_{f}}{\Gamma (\alpha )}\int_{0}^{t}\mathcal{N}_{\psi }^{\alpha }(t,s)\left[
\left[ \psi (s)-\psi (0)\right] ^{\gamma -1}\left\Vert u-v\right\Vert
_{C_{1-\gamma ;\psi }\left[ 0,b\right] }\right] ds \\
&&+\frac{2L_{f}L_{h}}{\zeta \gamma }\frac{1}{\Gamma (\alpha )}\int_{0}^{t}%
\mathcal{N}_{\psi }^{\alpha }(t,s)\text{ }\left[ \left[ \psi (s)-\psi (0)%
\right] ^{\gamma }\left\Vert u-v\right\Vert _{C_{1-\gamma ;\psi }\left[ 0,b%
\right] }\right] ds \\
&=&2L_{f}\left\Vert u-v\right\Vert _{C_{1-\gamma ;\psi }\left[ 0,b\right]
}I_{0^{+}}^{\alpha ;\psi }\left[ \psi (t)-\psi (0)\right] ^{\gamma -1} \\
&&+\frac{2L_{f}L_{h}}{\zeta \gamma }\left\Vert u-v\right\Vert _{C_{1-\gamma
;\psi }\left[ 0,b\right] }I_{0^{+}}^{\alpha ;\psi }\left[ \psi (t)-\psi (0)%
\right] ^{\gamma } \\
&=&2L_{f}\left\Vert u-v\right\Vert _{C_{1-\gamma ;\psi }\left[ 0,b\right] }%
\frac{\Gamma (\gamma )}{\Gamma (\gamma +\alpha )}\left[ \psi (t)-\psi (0)%
\right] ^{\alpha +\gamma -1} \\
&&+\frac{2L_{f}L_{h}}{\zeta \gamma }\left\Vert u-v\right\Vert _{C_{1-\gamma
;\psi }\left[ 0,b\right] }\frac{\Gamma (\gamma +1)}{\Gamma (\gamma +\alpha
+1)}\left[ \psi (t)-\psi (0)\right] ^{\alpha +\gamma }.
\end{eqnarray*}

From the definition of beta function, it follows that
\begin{eqnarray}
\left\Vert \mathcal{G}_{f}u-\mathcal{G}_{f}v\right\Vert _{C_{1-\gamma ;\psi }%
\left[ 0,b\right] } &\leq &2L_{f}\left( \frac{\mathcal{B}(\gamma ,\alpha )}{%
\Gamma (\alpha )}+\frac{L_{h}}{b\gamma }\frac{\mathcal{B}(\gamma +1,\alpha )%
}{\Gamma (\alpha )}\right)  \notag \\
&&\times \left[ \psi (b)-\psi (0)\right] ^{\alpha +1}\left\Vert
u-v\right\Vert _{C_{1-\gamma ;\psi }\left[ 0,b\right] }.  \label{r2}
\end{eqnarray}%
The condition (H$_{3}$) shows that $\mathcal{G}_{f}$ is a contraction
mapping on $C[-r,b],$ via the norm\ $\left\Vert \cdot \right\Vert
_{C_{1-\gamma ;\psi }\left[ 0,b\right] }$. An application the Banach
contraction principle shows that the problem (\ref{equ 1}) has a unique
solution in $C[-r,b]\cap C_{1-\gamma ;\psi }[0,b].$

Now we prove our second claim (2). Let $\varepsilon >0,$ and let $v\in
C[-r,b]\cap C_{1-\gamma ;\psi }\left[ 0,b\right] $ be a function which
satisfies the inequality (\ref{equ 5}). We denote by $u\in C[-r,b]\cap
C_{1-\gamma ;\psi }\left[ 0,b\right] $ the unique solution to the problem%
\begin{equation*}
\left\{
\begin{array}{c}
^{H}D_{0^{+}}^{\alpha ,\beta ;\psi }u(t)=F_{u,g,h}(t),\text{ \ }t\in
(0,b],\qquad \ \ \ \  \\
I_{0^{+}}^{1-\gamma ;\psi }u(0^{+})=I_{0^{+}}^{1-\gamma ;\psi
}v(0^{+}),\qquad \qquad \ \ \ \ \ \  \\
u(t)=v(t),\qquad t\in \lbrack -r,0],\ 0<r<\infty ,%
\end{array}%
\right.
\end{equation*}

Now, by using our first claim (1),%
\begin{equation*}
u(t)=\left\{
\begin{array}{c}
v(t)\ ,\qquad \qquad \qquad \qquad \qquad \qquad \ \qquad \ t\in \left[ -r,0%
\right] , \\
\mathcal{H}_{\psi }^{\gamma }(t,0)u_{0}+\frac{1}{\Gamma (\alpha )}%
\int_{0}^{t}\mathcal{N}_{\psi }^{\alpha }(t,s)\mathcal{F}_{u}(s)ds,\text{ \ }%
t\in \left( 0,b\right] ,%
\end{array}%
\right.
\end{equation*}%
Obviously, for $t\in (0,b]$ the Remark \ref{h1} gives,
\begin{eqnarray}
&&\left\vert v(t)-\mathcal{H}_{\psi }^{\gamma }(t,0)u_{0}-\frac{1}{\Gamma
(\alpha )}\int_{0}^{t}\mathcal{N}_{\psi }^{\alpha
}(t,s)F_{v,g,h}(s)ds\right\vert  \notag \\
&\leq &\varepsilon E_{\alpha }(\left[ \psi (t)-\psi (0)\right] ^{\alpha }).
\label{equ10}
\end{eqnarray}%
Note that, for all$\ t\in \lbrack -r,0],$ $\left\vert v(t)-u(t)\right\vert
=0 $.

Now, for all $t\in (0,b]$, it follows from (H$_{2}$) and (\ref{equ10}) that%
\begin{eqnarray}
\left\vert v(t)-u(t)\right\vert &\leq &\left\vert v(t)-\mathcal{H}_{\psi
}^{\gamma }(t,0)u_{0}-\frac{1}{\Gamma (\alpha )}\int_{0}^{t}\mathcal{N}%
_{\psi }^{\alpha }(t,s)F_{v,g,h}(s)ds\right\vert  \notag \\
&&+\left\vert \frac{1}{\Gamma (\alpha )}\int_{0}^{t}\mathcal{N}_{\psi
}^{\alpha }(t,s)F_{v,g,h}(s)ds-\frac{1}{\Gamma (\alpha )}\int_{0}^{t}%
\mathcal{N}_{\psi }^{\alpha }(t,s)F_{u,g,h}(s)ds\right\vert  \notag \\
&\leq &\varepsilon E_{\alpha }(\left[ \psi (t)-\psi (0)\right] ^{\alpha })
\notag \\
&&+\frac{1}{\Gamma (\alpha )}\int_{0}^{t}\mathcal{N}_{\psi }^{\alpha
}(t,s)\left\vert F_{v,g,h}(s)-F_{u,g,h}(s)\right\vert ds  \notag \\
&\leq &\varepsilon E_{\alpha }(\left[ \psi (t)-\psi (0)\right] ^{\alpha })
\notag \\
&&+L_{f}\frac{1}{\Gamma (\alpha )}\int_{0}^{t}\mathcal{N}_{\psi }^{\alpha
}(t,s)\bigg\{\left\vert u(s)-v(s)\right\vert +\left\vert
u(g(s)-v(g(s)\right\vert  \notag \\
&&+L_{h}\int_{0}^{s}\big(\left\vert u(\tau )-v(\tau )\right\vert +\left\vert
u(g(\tau )-v(g(\tau )\right\vert \big)d\tau \bigg\}ds.  \label{w1}
\end{eqnarray}%
In view of (\ref{w1}), for $z\in C([-r,b],%
\mathbb{R}
^{+})$ we consider the operator $\mathcal{T}:C([-r,b],%
\mathbb{R}
^{+})\rightarrow C([-r,b],%
\mathbb{R}
^{+})$ defined by%
\begin{equation*}
\mathcal{T}z(t)=0,\ \ \ t\in \left[ -r,0\right] ,
\end{equation*}%
\begin{eqnarray*}
\mathcal{T}z(t) &=&\varepsilon E_{\alpha }(\left[ \psi (t)-\psi (0)\right]
^{\alpha })+\frac{L_{f}}{\Gamma (\alpha )}\int_{0}^{t}\mathcal{N}_{\psi
}^{\alpha }(t,s)\bigg\{z(s)+z(g(s)) \\
&&+L_{h}\int_{0}^{s}z(\tau )+z(g(\tau ))d\tau \bigg\}ds,
\end{eqnarray*}%
for $t\in \left( 0,b\right] .$ We prove that $\mathcal{T}$\ is a Picard
operator.

Case 1. Observe first that for any $z,w\in C([-r,b],%
\mathbb{R}
^{+}),$
\begin{equation*}
\left\vert \mathcal{T}z(t)-\mathcal{T}w(t)\right\vert =0,\ \ \ t\in \left[
-r,0\right] ,
\end{equation*}%
Case 2. For each $t\in \left( 0,b\right] $ and $z,w\in C_{1-\gamma ;\psi }%
\left[ 0,b\right] ,$ it follows from ($H_{2}$) that%
\begin{eqnarray*}
\left\vert \mathcal{T}z(t)-\mathcal{T}w(t)\right\vert  &\leq &\frac{L_{f}}{%
\Gamma (\alpha )}\int_{0}^{t}\mathcal{N}_{\psi }^{\alpha }(t,s)\bigg\{%
\left\vert z(s)-w(s)\right\vert +\left\vert z(g(s))-w(g(s))\right\vert  \\
&&+L_{h}\int_{0}^{s}\big(\left\vert z(\tau )-w(\tau )\right\vert +\left\vert
z(g(\tau ))-w(g(\tau ))\right\vert \big)d\tau \bigg\}ds \\
&\leq &\frac{2L_{f}}{\Gamma (\alpha )}\int_{0}^{t}\mathcal{N}_{\psi
}^{\alpha }(t,s)\left[ \left[ \psi (s)-\psi (0)\right] ^{\gamma
-1}\left\Vert u-v\right\Vert _{C_{1-\gamma ;\psi }\left[ 0,b\right] }\right]
ds \\
&&+\frac{2L_{f}L_{h}}{\zeta \gamma }\frac{1}{\Gamma (\alpha )}\int_{0}^{t}%
\mathcal{N}_{\psi }^{\alpha }(t,s)\text{ }\left[ \left[ \psi (s)-\psi (0)%
\right] ^{\gamma }\left\Vert u-v\right\Vert _{C_{1-\gamma ;\psi }\left[ 0,b%
\right] }\right] ds \\
&=&2L_{f}\left( \frac{\mathcal{B}(\gamma ,\alpha )}{\Gamma (\alpha )}+\frac{%
L_{h}}{\zeta \gamma }\frac{\mathcal{B}(\gamma +1,\alpha )}{\Gamma (\alpha )}%
\right) \left[ \psi (t)-\psi (0)\right] ^{\alpha +\gamma }\left\Vert
u-v\right\Vert _{C_{1-\gamma ;\psi }\left[ 0,b\right] }.
\end{eqnarray*}%
Then we obtain%
\begin{eqnarray*}
\left\Vert \mathcal{T}z-\mathcal{T}w\right\Vert _{C_{1-\gamma ,\psi \lbrack
0,b]}} &\leq &2L_{f}\left( \frac{\mathcal{B}(\gamma ,\alpha )}{\Gamma
(\alpha )}+\frac{L_{h}}{b\gamma }\frac{\mathcal{B}(\gamma +1,\alpha )}{%
\Gamma (\alpha )}\right)  \\
&&\times \left[ \psi (b)-\psi (0)\right] ^{\alpha +1}\left\Vert
u-v\right\Vert _{C_{1-\gamma ;\psi }\left[ 0,b\right] },
\end{eqnarray*}%
By (H$_{3}$), $\mathcal{T}$ is a contraction mapping on $C([-r,b],%
\mathbb{R}
^{+})$ via the wieghted norm\ $\left\Vert .\right\Vert _{C_{1-\gamma ;\psi }%
\left[ 0,b\right] }.$ Applying the Banach contraction principle to $\mathcal{%
T}$, we see that $\mathcal{T}$ is a Picard operator and $F_{\mathcal{T}%
}=\{z^{\ast }\}$. Then, for all$\ t\in (0,b]$, we have%
\begin{eqnarray*}
z^{\ast }(t) &=&\mathcal{T}z^{\ast }(t)) \\
&=&\varepsilon E_{\alpha }(\left[ \psi (t)-\psi (0)\right] ^{\alpha })+\frac{%
L_{f}}{\Gamma (\alpha )}\int_{0}^{t}\mathcal{N}_{\psi }^{\alpha }(t,s)\bigg\{%
z^{\ast }(s)+z^{\ast }(g(s)) \\
&&+L_{h}\int_{0}^{s}z^{\ast }(\tau )+z^{\ast }(g(\tau ))d\tau \bigg\}ds,
\end{eqnarray*}%
Next, we prove that the solution $z^{\ast }$ is increasing. Let $\sigma
:=\min_{s\in \lbrack 0,b]}[z^{\ast }(s)+z^{\ast }(h(s))]\in
\mathbb{R}
_{+}.$Then for all $0<t_{1}<t_{2}\leq b$, we have%
\begin{eqnarray*}
z^{\ast }(t_{2})-z^{\ast }(t_{1}) &=&\varepsilon E_{\alpha }(\left[ \psi
(t_{2})-\psi (0)\right] ^{\alpha })-\varepsilon E_{\alpha }(\left[ \psi
(t_{1})-\psi (0)\right] ^{\alpha }) \\
&&+\frac{L_{f}}{\Gamma (\alpha )}\int_{0}^{t_{1}}\big(\mathcal{N}_{\psi
}^{\alpha }(t_{2},s)-\mathcal{N}_{\psi }^{\alpha }(t_{1},s)\big)\bigg\{%
z^{\ast }(s)+z^{\ast }(g(s)) \\
&&+L_{h}\int_{0}^{s}z^{\ast }(\tau )+z^{\ast }(g(\tau ))d\tau \bigg\}ds, \\
&&+\frac{L_{f}}{\Gamma (\alpha )}\int_{t_{1}}^{t_{2}}\mathcal{N}_{\psi
}^{\alpha }(t_{2},s)\bigg\{z^{\ast }(s)+z^{\ast }(g(s)) \\
&&+L_{h}\int_{0}^{s}z^{\ast }(\tau )+z^{\ast }(g(\tau ))d\tau \bigg\}ds \\
&\geq &\varepsilon E_{\alpha }(\left[ \psi (t_{2})-\psi (0)\right] ^{\alpha
})-\varepsilon E_{\alpha }(\left[ \psi (t_{1})-\psi (0)\right] ^{\alpha }) \\
&&+\frac{L_{f}}{\Gamma (\alpha )}\int_{0}^{t_{1}}\big(\mathcal{N}_{\psi
}^{\alpha }(t_{2},s)-\mathcal{N}_{\psi }^{\alpha }(t_{1},s)\big)\sigma \big(%
1+L_{h}s\big)ds \\
&&+\frac{L_{f}}{\Gamma (\alpha )}\int_{t_{1}}^{t_{2}}\mathcal{N}_{\psi
}^{\alpha }(t_{2},s)\big)\sigma \big(1+L_{h}s\big)ds \\
&=&\varepsilon E_{\alpha }(\left[ \psi (t_{2})-\psi (0)\right] ^{\alpha
})-\varepsilon E_{\alpha }(\left[ \psi (t_{1})-\psi (0)\right] ^{\alpha }) \\
&&+\frac{\sigma L_{f}}{\Gamma (\alpha +1)}\left[ (\psi (t_{2})-\psi
(0))^{\alpha }-\psi (t_{1})-\psi (0))^{\alpha }\right]  \\
&&+\frac{\sigma L_{f}L_{h}}{\Gamma (\alpha +2)\zeta }\left[ (\psi
(t_{2})-\psi (0))^{\alpha +1}-(\psi (t_{1})-\psi (0))^{\alpha +1}\right]  \\
&>&0,
\end{eqnarray*}%
Therefore, $u^{\ast }$ is increasing, so $z^{\ast }(g(t))\leq z^{\ast }(t)\ $%
due to $g(t)\leq t$ and%
\begin{eqnarray*}
z^{\ast }(t) &\leq &\varepsilon E_{\alpha }(\left[ \psi (t)-\psi (0)\right]
^{\alpha }) \\
&&+\int_{0}^{t}\frac{2L_{f}}{\Gamma (\alpha )}\mathcal{N}_{\psi }^{\alpha
}(t,s)\bigg\{z^{\ast }(s)+\int_{0}^{s}L_{h}z^{\ast }(\tau )d\tau \bigg\}ds.
\end{eqnarray*}

Applying Pachpatte's inequality given in the Lemma \ref{tt} to the
inequality (\ref{ht}) with $x(t)=z^{\ast }(t),$ $\eta (t)=\varepsilon
E_{\alpha }(\left[ \psi (t)-\psi (0)\right] ^{\alpha }),$ $p(s)=\frac{2L_{f}%
}{\Gamma (\alpha )}\mathcal{N}_{\psi }^{\alpha }(t,s)$ and $q(\sigma )=L_{h},
$ we obtain%
\begin{eqnarray*}
z^{\ast }(t) &\leq &\varepsilon E_{\alpha }(\left[ \psi (t)-\psi (0)\right]
^{\alpha }) \\
&&\times \bigg{(}1+\int_{0}^{t}\frac{2L_{f}}{\Gamma (\alpha )}\mathcal{N}%
_{\psi }^{\alpha }(t,s)\exp \left\{ \int_{0}^{s}\bigg{(}\frac{2L_{f}}{\Gamma
(\alpha )}\mathcal{N}_{\psi }^{\alpha }(s,\tau )+L_{h}\bigg{)}d\tau \right\}
ds\bigg{)} \\
&\leq &\varepsilon E_{\alpha }(\left[ \psi (t)-\psi (0)\right] ^{\alpha }) \\
&&\times \bigg{(}1+\int_{0}^{t}\frac{2L_{f}}{\Gamma (\alpha )}\mathcal{N}%
_{\psi }^{\alpha }(t,s)\exp \left\{ \frac{2L_{f}}{\Gamma (\alpha +1)}(\psi
(s)-\psi (0))^{\alpha }+L_{h}s\right\} ds\bigg{)}.
\end{eqnarray*}

By Lagrange Mean value theorem, there exist $c\in (0,s]$ such that $(\psi
(s)-\psi (0))=s\psi ^{\prime }(c),$ it follows from fact that $\psi \in
C^{1}[0,b],$ there exists a constant $\kappa \neq 0$ such that $\sup_{\tau
\in (0,c]}\left\vert \psi ^{\prime }(\tau )\right\vert =\kappa ,$ and we
have from $0<\alpha <1,$ that%
\begin{equation*}
\left[ \psi (s)-\psi (0)\right] ^{\alpha }<\left[ \psi (s)-\psi (0)\right] ,%
\text{ and }s=\frac{1}{\kappa }\left[ \psi (s)-\psi (0)\right] .
\end{equation*}

This gives%
\begin{eqnarray*}
z^{\ast }(t) &\leq &\varepsilon E_{\alpha }(\left[ \psi (t)-\psi (0)\right]
^{\alpha }) \\
&&\times \bigg{(}1+\int_{0}^{t}\frac{2L_{f}}{\Gamma (\alpha )}\mathcal{N}%
_{\psi }^{\alpha }(t,s)\exp \left\{ \left( \frac{2L_{f}}{\Gamma (\alpha +1)}+%
\frac{L_{h}}{\kappa }\right) \left[ \psi (s)-\psi (0)\right] \right\} ds%
\bigg{)} \\
&=&\varepsilon E_{\alpha }(\left[ \psi (t)-\psi (0)\right] ^{\alpha })%
\bigg{(}1+2L_{f}I_{0^{+}}^{\alpha ;\psi }e^{A\left[ \psi (t)-\psi (0)\right]
}\bigg{)} \\
&\leq &\varepsilon E_{\alpha }(\left[ \psi (t)-\psi (0)\right] ^{\alpha })%
\bigg{(}1+2L_{f}\left[ \psi (b)-\psi (0)\right] ^{\alpha }E_{1,\alpha +1}(A%
\left[ \psi (b)-\psi (0)\right] )\bigg{)}
\end{eqnarray*}%
where $A:=\left( \frac{2L_{f}}{\Gamma (\alpha +1)}+\frac{L_{h}}{\kappa }%
\right) .$ Take
\begin{equation*}
C_{E_{\alpha }}=\bigg{(}1+2L_{f}\left[ \psi (b)-\psi (0)\right] ^{\alpha
}E_{1,\alpha +1}(A\left[ \psi (b)-\psi (0)\right] )\bigg{)},
\end{equation*}%
we get%
\begin{equation*}
z^{\ast }(t)\leq C_{E_{\alpha }}\varepsilon E_{\alpha }(\left[ \psi (t)-\psi
(0)\right] ^{\alpha }),
\end{equation*}%
In particular, if $z=|v-u|$, from (\ref{w1}), $z\leq \mathcal{T}z$ and
applying the Lemma \ref{le2.6}, we obtain $z\leq z^{\ast }$, where $\mathcal{%
T}$ is an increasing Picard operator. As a result, we get%
\begin{equation*}
\left\vert v(t)-u(t)\right\vert \leq C_{E_{\alpha }}\varepsilon E_{\alpha }(
\left[ \psi (t)-\psi (0)\right] ^{\alpha }),\ \ \ t\in \left[ -r,b\right] .
\end{equation*}%
Thus, the first equation of (\ref{equ 1}) is UHML stable.
\end{proof}

Next, we use the Bielecki's norm

\begin{equation*}
\left\Vert \omega \right\Vert _{B}:=\underset{t\in \left[ a,b\right] }{\max }%
e^{-\delta \left[ \psi (t)-\psi (a)\right] }\left\vert \left[ \psi (t)-\psi
(a)\right] ^{1-\gamma }\omega (t)\right\vert ,\text{ }\delta \geq 0,
\end{equation*}%
where
\begin{equation*}
B=\left\{ \omega :(a,b]\rightarrow
\mathbb{R}
;\text{ }e^{-\delta \left[ \psi (t)-\psi (a)\right] }\left[ \psi (t)-\psi (a)%
\right] ^{1-\gamma }\omega (t)\in C\left[ a,b\right] \right\} .
\end{equation*}

\begin{theorem}
\label{th3.3} Assume that $(H_{1})$- $(H_{3})$ are satisfied. If we have the
inequality
\begin{equation}
2L_{f}e^{\delta \left[ \psi (b)-\psi (0)\right] }\frac{\Gamma (\gamma )}{%
\Gamma (\gamma +\alpha )}\left( 1+\frac{L_{h}}{\zeta (\gamma +\alpha )}%
\right) \left[ \psi (b)-\psi (0)\right] ^{\alpha +1}<1.  \label{EE1}
\end{equation}

Then

\begin{enumerate}
\item The $\psi -$Hilfer problem (\ref{equ 1}) has a unique solution in $%
C[-r,b]\cap C_{1-\gamma ;\psi }[0,b]$.

\item The first equation of (\ref{equ 1}) is UHML stable.
\end{enumerate}

\begin{proof}
Just like the debate in Theorem \ref{Th3.1}, we only show that we show that $%
\mathcal{G}_{f}:C[-r,b]\rightarrow C[-r,b]$ defined by (\ref{equ9}) is a
contracting mapping on $C[-r,b]$\ with respect to\ the Bielecki's norm $%
\left\Vert \cdot \right\Vert _{B}.$ Since the procedure is standard, we only
present the main variation in the proof as follows:

For each $t\in \lbrack -r,0]$ and for each $u,v\in C\left[ -r,b\right] .$we
have%
\begin{equation*}
\left\vert \mathcal{G}_{f}(u)(t)-\mathcal{G}_{f}(v)(t)\right\vert =0.
\end{equation*}%
On the other hand, for each $u,v\in B$ and for all $t\in (0,b]$, we have%
\begin{eqnarray}
\left\vert \mathcal{G}_{f}(u)(t)-\mathcal{G}_{f}(v)(t)\right\vert  &\leq &%
\frac{1}{\Gamma (\alpha )}\int_{0}^{t}\mathcal{N}_{\psi }^{\alpha
}(t,s)\left\vert F_{u,g,h}(s)-F_{u,g,h}(s)\right\vert ds  \notag \\
&\leq &\frac{2L_{f}}{\Gamma (\alpha )}\int_{0}^{t}\mathcal{N}_{\psi
}^{\alpha }(t,s)e^{\delta \left[ \psi (s)-\psi (0)\right] }\left[ \left[
\psi (s)-\psi (0)\right] ^{\gamma -1}\left\Vert u-v\right\Vert _{B}\right] ds
\notag \\
&&+\frac{2L_{f}L_{h}}{\zeta \gamma \Gamma (\alpha )}\int_{0}^{t}\mathcal{N}%
_{\psi }^{\alpha }(t,s)e^{\delta \left[ \psi (s)-\psi (0)\right] }\left[ %
\left[ \psi (s)-\psi (0)\right] ^{\gamma }\left\Vert u-v\right\Vert _{B}%
\right] ds  \notag \\
&=&\left( 2L_{f}J_{1}+\frac{2L_{f}L_{h}}{\zeta \gamma }J_{2}\right)
\left\Vert u-v\right\Vert _{B},  \label{rt}
\end{eqnarray}%
where%
\begin{eqnarray*}
J_{1} &:&=\frac{1}{\Gamma (\alpha )}\int_{0}^{t}\mathcal{N}_{\psi }^{\alpha
}(t,s)e^{\delta \left[ \psi (s)-\psi (0)\right] }\left[ \psi (s)-\psi (0)%
\right] ^{\gamma -1}ds, \\
J_{2} &:&=\frac{1}{\Gamma (\alpha )}\int_{0}^{t}\mathcal{N}_{\psi }^{\alpha
}(t,s)e^{\delta \left[ \psi (s)-\psi (0)\right] }\left[ \psi (s)-\psi (0)%
\right] ^{\gamma }ds.
\end{eqnarray*}%
We also estimate $J_{1}$, $J_{2}$ terms separately. By Theorem (\ref{th2.3a}%
) we have%
\begin{equation}
J_{1}\leq e^{\delta \left[ \psi (b)-\psi (0)\right] }\frac{\Gamma (\gamma )}{%
\Gamma (\gamma +\alpha )}\left[ \psi (t)-\psi (0)\right] ^{\alpha +\gamma
-1},  \label{R1}
\end{equation}%
and
\begin{equation}
J_{2}\leq e^{\delta \left[ \psi (b)-\psi (0)\right] }\frac{\Gamma (\gamma +1)%
}{\Gamma (\gamma +\alpha +1)}\left[ \psi (t)-\psi (0)\right] ^{\alpha
+\gamma }.  \label{R2}
\end{equation}%
Equation (\ref{rt}) with (\ref{R1}) and (\ref{R2}), lead to
\begin{eqnarray*}
\left\Vert \mathcal{G}_{f}u-\mathcal{G}_{f}v\right\Vert _{C_{1-\gamma ;\psi }%
\left[ 0,b\right] } &\leq &2L_{f}e^{\delta \left[ \psi (b)-\psi (0)\right] }%
\frac{\Gamma (\gamma )}{\Gamma (\gamma +\alpha )}\left[ \psi (b)-\psi (0)%
\right] ^{\alpha }\left\Vert u-v\right\Vert _{B} \\
&&+\frac{2L_{f}L_{h}}{\zeta \gamma }e^{\delta \left[ \psi (b)-\psi (0)\right]
}\frac{\Gamma (\gamma +1)}{\Gamma (\gamma +\alpha +1)}\left[ \psi (b)-\psi
(0)\right] ^{\alpha +1}\left\Vert u-v\right\Vert _{B} \\
&\leq &2L_{f}e^{\delta \left[ \psi (b)-\psi (0)\right] }\frac{\Gamma (\gamma
)}{\Gamma (\gamma +\alpha )}\left( 1+\frac{L_{h}}{\zeta (\gamma +\alpha )}%
\right) \left[ \psi (b)-\psi (0)\right] ^{\alpha +1}\left\Vert
u-v\right\Vert _{B}.
\end{eqnarray*}%
By the inequality (\ref{EE1}), $\mathcal{G}_{f}$ is a contraction mapping on
$C[-r,b],$ via the Bielecki's norm\ $B$. An application the Banach
contraction principle shows that $\psi -$ Hilfer problem (\ref{equ 1}) has a
unique solution in $C\left( [-r,b],%
\mathbb{R}
\right) \cap C_{1-\gamma ;\psi }\left( [0,b],%
\mathbb{R}
\right) .$

The proof of UHML stability is just like in Theorem \ref{Th3.1} so we omit
it here.
\end{proof}
\end{theorem}

\begin{remark}
\qquad

\begin{enumerate}
\item If $\int_{0}^{t}h(t,s,y(s),y(g(s))ds=0,$ then problem (\ref{equ 1})
reduces to the problem (\ref{lwo}) in \cite{LWO}.

\item If $\int_{0}^{t}h(t,s,y(s),y(g(s))ds=0,$ and $\beta =1,$ then problem (%
\ref{equ 1}) reduces to the problem (\ref{wz}) in \cite{WZ}.

\item If $\alpha =\beta =1,$ then problem (\ref{equ 1}) reduces to the
problem (\ref{ku1}) in \cite{KU1}.
\end{enumerate}
\end{remark}

\section{\textbf{An example \label{Sec5}}}

Will be provided in the revised submission.

\section{Conclusion}

We have obtained some existence, uniqueness and Ulam--Hyers--Mittag-Leffler
(UHML) stability results for the solution of Cauch type problem for $\psi -$%
Hilfer FFIDEs based on the reduction of fractional differential equations
(FDEs) to integral equations. The employed techniques the Picard operator
method and generalized Pachpatte's inequality are quite general and
effective. We trust the reported results here will have a positive impact on
the development of further applications in engineering and applied sciences.

\end{document}